\documentclass[11pt]{amsart}

\usepackage{amsmath,amssymb,amsfonts,amsthm,mathtools}
\usepackage{array,booktabs,longtable,tabularx}
\usepackage{xcolor}
\usepackage{tikz}
\usetikzlibrary{matrix,decorations.pathreplacing,calc,positioning,backgrounds}
\usepackage[top=1in,bottom=1in,left=1in,right=1in]{geometry}
\usepackage[colorlinks=true,linkcolor=blue,citecolor=blue,urlcolor=blue]{hyperref}

\theoremstyle{plain}
\newtheorem{theorem}{Theorem}[section]

\newtheorem{corollary}[theorem]{Corollary}

\newtheorem{proposition}[theorem]{Proposition}

\theoremstyle{definition}
\newtheorem{definition}[theorem]{Definition}
\newtheorem{exam}[theorem]{Example}

\title{Characteristic Independence of Betti Numbers of Monomial Ideals in Five Variables}
\author{Noah Ripke}
\address{Department of Mathematics, Yale University, 219 Prospect Street, New Haven, CT 06511, USA}
\email{noah.ripke@yale.edu}

\author{Phillip Yoon}
\address{Department of Mathematics, Hunter College, City University of New York, 695 Park Ave., New York, NY 10065, USA}
\email{phillip.yoon58@login.cuny.edu}

\begin{document}

\begin{abstract}
Alesandroni proved that Betti numbers of monomial ideals in at most four variables are independent of the characteristic of the base field, while characteristic-dependent Betti numbers occur in six variables. We prove that the five-variable case is characteristic-independent. More precisely, if $S=k[x_1,\ldots,x_5]$ and $M\subseteq S$ is a monomial ideal, then the multigraded, graded, and total Betti numbers of $S/M$ are independent of $\operatorname{char}(k)$. The proof reduces arbitrary monomial ideals to squarefree twin ideals and then applies Hochster's formula. The topological input is that simplicial complexes on at most five vertices have torsion-free integral homology. The six-variable example arising from the six-vertex triangulation of $\mathbb{RP}^{2}$ shows that the bound is sharp. We also record a computation of the graded Betti tables of squarefree monomial ideals in five variables up to relabeling.
\end{abstract}


\maketitle

\section{Introduction}

We investigate whether Betti numbers of monomial ideals in five variables can depend on the characteristic of the base field. Characteristic independence is known in certain fixed homological degrees without a restriction on the number of variables. In particular, Terai and Hibi used Alexander duality to prove that the total second Betti number of every Stanley--Reisner ring is independent of the choice of base field \cite{teraihibi1996alexander}. Alesandroni \cite{alesandroni2023betti} proved characteristic independence for monomial ideals in at most four variables. The six-variable example recorded by Peeva \cite{peeva2011graded} shows that characteristic dependence occurs in six variables. This leaves the five-variable case open. We prove that Betti numbers of monomial ideals in five variables are indeed characteristic independent:

\begin{theorem}\label{allfivevariabletheorem}
Let $S=k[x_1,\ldots,x_5]$, and let $M\subseteq S$ be a monomial ideal. Then the multigraded, graded, and total Betti numbers of $S/M$ are independent of the characteristic of $k$.
\end{theorem}
Thus six is the smallest number of variables in which characteristic dependence can occur.

The proof reduces the problem to the squarefree case. For each Taylor multidegree $m$, Alesandroni's squarefree twin ideal construction identifies $\beta^S_{i,m}(S/M)$ with a Betti number of a squarefree monomial ideal in at most five variables. Such an ideal is the Stanley--Reisner ideal of a simplicial complex on at most five vertices. By Hochster's formula, its Betti numbers are determined by the reduced homology of induced subcomplexes. Since every connected simplicial complex on at most five vertices is homotopy equivalent to a wedge of spheres \cite{govc2025fundamental}, these induced subcomplexes have torsion-free integral homology. The Universal Coefficient Theorem then implies that the relevant homology dimensions are independent of the characteristic.

Thus characteristic-dependence cannot occur for monomial ideals in at most five variables. The six-variable example coming from the six-vertex triangulation of $\mathbb{RP}^{2}$ shows that this bound is sharp: the triangulation has $\mathbb Z/2$-torsion in integral homology \cite{reisner1976cohen}.

\section{Reduction to Squarefree Ideals}\label{reductiontosquareidealssection}

The first step is to reduce arbitrary monomial ideals to squarefree monomial ideals. Let $S=k[x_1,\ldots,x_n]$, and let $M\subseteq S$ be a monomial ideal with minimal generating set $G(M)$. Let $\mathbb T_M$ denote the Taylor resolution of $S/M$. We call a monomial $m$ a Taylor multidegree of $M$ if
\[
m=\operatorname{lcm}(g_{j_1},\ldots,g_{j_t})
\]
for some subset $\{g_{j_1},\ldots,g_{j_t}\}\subseteq G(M)$, with the convention that the least common multiple of the empty subset is 1.

For a monomial $m$, define
\[
M_m=(g\in G(M): g\mid m).
\]
Equivalently, $M_m$ is the ideal generated by the minimal generators of $M$ that divide $m$. The following theorem of Gasharov--Hibi--Peeva allows us to ignore all generators of $M$ that do not divide the multidegree under consideration.

\begin{theorem}[Gasharov--Hibi--Peeva {\cite[Theorem~2.1]{GasharovHibiPeeva2002}}]\label{GHPmultidegreereduction}
Let $M$ be minimally generated by $G(M)$, and let $m$ be a monomial. Let
\[
M_m=(g\in G(M): g\mid m).
\]
Then
\[
\beta_{i,m}^S(S/M)=\beta_{i,m}^S(S/M_m)
\]
for every $i$.
\end{theorem}

Now write
\[
m=x_1^{\alpha_1}\cdots x_n^{\alpha_n}.
\]
For a generator
\[
g=x_1^{b_1}\cdots x_n^{b_n}
\]
of $M_m$, define
\[
g'=\prod_{\substack{1\leq j\leq n\\ b_j=\alpha_j}} x_j^{\alpha_j}.
\]
The ideal generated by the monomials $g'$, as $g$ ranges over the generators of $M_m$, is called the twin ideal of $M_m$ and is denoted $M_m'$.

To make this ideal squarefree, set
\[
T_m=k[y_j:\alpha_j>0]
\]
and define
\[
g''=\prod_{\substack{1\leq j\leq n\\ b_j=\alpha_j\\ \alpha_j>0}}y_j.
\]
The ideal generated by these monomials $g''$ is called the squarefree twin ideal of $M_m$ and is denoted $M_m''\subseteq T_m$. Finally, set
\[
y_m=\prod_{\alpha_j>0}y_j.
\]

\begin{theorem}[Squarefree twin reduction]\label{squarefreetwinidealreductionthm}
Let $M\subseteq S=k[x_1,\ldots,x_n]$ be a monomial ideal, and let $m$ be a Taylor multidegree of $M$. With the notation above,
\[
\beta_{i,m}^S(S/M)=\beta_{i,y_m}^{T_m}(T_m/M_m'')
\]
for every $i$.
\end{theorem}

\begin{proof}
By Theorem~\ref{GHPmultidegreereduction},
\[
\beta_{i,m}^S(S/M)=\beta_{i,m}^S(S/M_m).
\]
Since $m$ is a Taylor multidegree of $M$, the monomial $m$ is the least common multiple of the minimal generators of $M_m$. Hence Alesandroni's twin ideal theorem \cite{alesandroni2023betti} applies to $M_m$ in multidegree $m$ and gives
\[
\beta_{i,m}^S(S/M_m)=\beta_{i,m}^S(S/M_m')
\]
for every $i$.

Let
\[
S_m=k[x_j:\alpha_j>0].
\]
The variables $x_j$ with $\alpha_j=0$ do not occur in $M_m'$, so passing between $S_m$ and $S$ by polynomial extension does not change the multigraded Betti number
in multidegree $m$. Thus
\[
\beta_{i,m}^S(S/M_m')=\beta_{i,m}^{S_m}(S_m/M_m').
\]

Now consider the homomorphism
\[
\varphi:T_m=k[y_j:\alpha_j>0]\longrightarrow S_m,
\qquad
y_j\longmapsto x_j^{\alpha_j}.
\]
The ring $S_m$ is a free, hence flat, $T_m$-module via $\varphi$; explicitly, a
$T_m$-basis is given by the monomials
\[
\prod_{\alpha_j>0}x_j^{c_j}
\qquad
0\leq c_j<\alpha_j.
\]
Moreover, extension of scalars sends $M_m''$ to $M_m'$. Therefore a minimal multigraded free resolution of $T_m/M_m''$, after extension of scalars along $\varphi$, gives a minimal multigraded free resolution of $S_m/M_m'$, with the squarefree multidegree $y_m$ corresponding to the multidegree $m$. Hence
\[
\beta_{i,m}^{S_m}(S_m/M_m')
=
\beta_{i,y_m}^{T_m}(T_m/M_m'').
\]
Combining the displayed equalities gives
\[
\beta_{i,m}^S(S/M)=\beta_{i,y_m}^{T_m}(T_m/M_m'')
\]
for every $i$.
\end{proof}

\section{Stanley--Reisner Ideals}

Let $S=k[x_1,\ldots,x_n]$. For $F\subseteq[n]$, write
\[
x_F=\prod_{i\in F}x_i,
\]
with $x_\varnothing=1$. Under the correspondence $F\leftrightarrow x_F$, divisibility of squarefree monomials corresponds to inclusion of subsets.

\begin{definition}
An abstract simplicial complex $\Delta$ on vertex set contained in $[n]$ is a collection of subsets of $[n]$ closed under taking subsets. A subset $F\subseteq[n]$ is a nonface if $F\notin\Delta$, and it is a minimal nonface if every proper subset of $F$ lies in $\Delta$.
\end{definition}

\begin{definition}
The Stanley--Reisner ideal of $\Delta$ is
\[
I_\Delta=(x_F : F\subseteq[n],\ F\notin\Delta).
\]
Conversely, if $I\subseteq S$ is a squarefree monomial ideal, its Stanley--Reisner complex is
\[
\Delta_I=\{F\subseteq[n]:x_F\notin I\}.
\]
\end{definition}

Therefore, squarefree monomial ideals are precisely Stanley--Reisner ideals, and the minimal monomial generators of $I_\Delta$ correspond to the minimal nonfaces of $\Delta$.

\begin{definition}
For $W\subseteq[n]$, the induced subcomplex of $\Delta$ on $W$ is
\[
\Delta_W=\{F\in\Delta:F\subseteq W\}.
\]
\end{definition}

\begin{exam}
Let
\[
I=(x_1x_3,\ x_2x_4)\subseteq k[x_1,x_2,x_3,x_4].
\]
The corresponding Stanley--Reisner complex is
\[
\Delta_I=\{F\subseteq [4]:x_F\notin I\}.
\]
In other words, a subset $F\subseteq[4]$ is a face of $\Delta_I$ exactly when it does not contain $\{1,3\}$ and does not contain $\{2,4\}$. Recall that $x_F$ is the squarefree monomial corresponding to the subset $F$.

The sets $\{1,3\}$ and $\{2,4\}$ are nonfaces because
\[
x_1x_3\in I
\quad\text{and}\quad
x_2x_4\in I.
\]
They are minimal nonfaces because all of their proper subsets are faces. For example, the proper subsets of $\{1,3\}$ are $\varnothing, \{1\}, \{3\}$, and none of the corresponding monomials $ 1, x_1,x_3$ lies in $I$. Similarly, the proper subsets of $\{2,4\}$ are faces.

On the other hand, $\{1,2,3\}$ is also a nonface, since it contains $\{1,3\}$ and
\[
x_1x_2x_3\in I.
\]
But $\{1,2,3\}$ is not a minimal nonface because it has the proper subset $\{1,3\}$ which is already a nonface.

Thus the minimal nonfaces of $\Delta_I$ are exactly
\[
\{1,3\},\quad \{2,4\},
\]
and these correspond to the minimal generators
\[
x_1x_3,\quad x_2x_4
\]
of $I$.
\end{exam}

\section{Hochster's Formula and the Torsion Criterion}\label{Hochstersformulasection}

We use reduced simplicial homology with integer coefficients. For a finite simplicial complex $\Delta$, we write $\widetilde H_q(\Delta;\mathbb Z)$ for its $q$th reduced homology group. Recall that for $q\geq 1$, reduced homology agrees with ordinary homology, while
\[
\widetilde H_0(\Delta;\mathbb Z)\cong \mathbb Z^{c-1}
\]
when $\Delta$ has $c$ connected components.

We use the standard convention
\[
\widetilde H_{-1}(\Gamma;\mathbb Z)\cong
\begin{cases}
\mathbb Z, & \Gamma=\{\varnothing\},\\
0, & \text{otherwise},
\end{cases}
\quad
\widetilde H_q(\Gamma;\mathbb Z)=0\quad(q<-1).
\]
The same convention is used with coefficients in a field $k$. This convention is needed in Hochster's formula when an induced subcomplex has no nonempty faces.

\begin{definition}[Torsion]
An abelian group $A$ has torsion if there exists a nonzero element $a\in A$ and a positive integer $r$ such that $ra=0$. We say that a simplicial complex has torsion in reduced integral homology if some $\widetilde H_q(\Delta;\mathbb Z)$ has torsion.
\end{definition}

The following standard form of the Universal Coefficient Theorem can be found in
\cite[Section~3.A]{hatcher}:

\begin{theorem}[Universal Coefficient Theorem]\label{UniversalCoefficientTheorem}
Let $\Delta$ be a finite simplicial complex and let $k$ be a field. For every $q \geq 0$, there is a short exact sequence
\[
0 \to \widetilde H_q(\Delta; \mathbb Z) \otimes_{\mathbb Z} k \to \widetilde H_q (\Delta;k) \to \operatorname{Tor}_1^{\mathbb Z} \left(\widetilde H_{q-1}(\Delta; \mathbb Z),k \right) \to 0.
\]
\end{theorem}

The $\operatorname{Tor}$-term is the possible source of characteristic-dependence: if $\widetilde H_{q-1}(\Delta;\mathbb Z)$ is torsion-free, then this term vanishes. When $q=0$, our convention gives $\widetilde H_{-1}(\Delta;\mathbb Z)\cong 0$ or $\mathbb Z$, so no torsion arises from degree -1.

\begin{corollary}\label{UniversalCoefficientTheoremCorollary}
If $\widetilde H_q (\Delta; \mathbb Z)$ is torsion-free for every $q$, then 
\[
\dim_k \widetilde H_q (\Delta;k) = \operatorname{rank}_{\mathbb Z} \widetilde H_q (\Delta; \mathbb Z)
\]
for every field $k$. In particular, this dimension is independent of $\operatorname{char}(k)$.
\end{corollary}

\begin{proof}
Fix $q\geq0$. Since $\widetilde H_{q-1} (\Delta ; \mathbb Z)$ is torsion-free by assumption, the $\operatorname{Tor}$-term becomes 0. Consequently, by Theorem~\ref{UniversalCoefficientTheorem},
\[
\widetilde H_q(\Delta ; k) \cong \widetilde H_q(\Delta ; \mathbb Z) \otimes_{\mathbb Z} k.
\]

Because $\Delta$ is finite, $\widetilde H_q (\Delta ; \mathbb Z)$ is a finitely generated abelian group. Since it is torsion-free, it is free abelian, say 
\[
\widetilde H_q (\Delta; \mathbb Z) \cong \mathbb Z^r. 
\]
Therefore,
\[
\widetilde H_q (\Delta;k) \cong \mathbb Z^r \otimes_{\mathbb Z} k \cong k^r.
\]
Thus, \[
\dim_k \widetilde H_q(\Delta;k) = r = \operatorname{rank}_{\mathbb Z}\widetilde H_q (\Delta; \mathbb Z), 
\]
which is independent of $\operatorname{char}(k)$.
\end{proof}

\begin{theorem}[Hochster's Formula {\cite[Corollary~5.12]{MillerSturmfels2005}}]\label{Hochstersformula}
Let $I_\Delta\subseteq S=k[x_1,\dots,x_n]$ be a Stanley--Reisner ideal. For
$W\subseteq[n]$, write
\[
x_W=\prod_{r\in W}x_r.
\]
Then the multigraded Betti numbers are given by
\[
\beta_{i,x_W}^S(S/I_\Delta)
=
\dim_k \widetilde H_{|W|-i-1}(\Delta_W;k).
\]
Consequently, the graded Betti numbers are given by
\[
\beta_{i,j}^S(S/I_\Delta)
=
\sum_{\substack{W\subseteq[n]\\ |W|=j}}
\dim_k \widetilde H_{j-i-1}(\Delta_W;k).
\]

Moreover, $\beta_{i,\mathbf a}^S(S/I_\Delta)=0$ if $\mathbf a\notin\{0,1\}^n$.
\end{theorem}

\begin{exam}\label{hollowtetrahedronexample}
Consider the simplicial complex $\Delta$ on vertices
$\{1,2,3,4,5\}$ whose facets are
\[
\{2,3,4\},\quad
\{2,3,5\},\quad
\{2,4,5\},\quad
\{3,4,5\},\quad
\{1,2\},\quad
\{1,4\}.
\]

Geometrically, this is the boundary of a hollow tetrahedron on vertices $\{2,3,4,5\}$, together with a hollow triangle on vertices $\{1,2,4\}$ attached along the edge $\{2,4\}$, as shown in Figure~\ref{fig:hollow-tetrahedron-attached-triangle}.

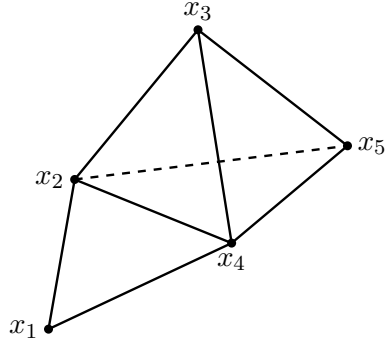
\begin{figure}[htbp]
\centering 
\begin{tikzpicture}[scale=.99, 
    vertex/.style={circle, fill=black, inner sep=1.2pt},
    edge/.style={line width=0.9pt},
    hidden/.style={line width=0.9pt, dashed}
]

\coordinate (x1) at (0,0);
\coordinate (x2) at (0.35,2.0);
\coordinate (x3) at (2.0,4.0);
\coordinate (x4) at (2.45,1.15);
\coordinate (x5) at (4.0,2.45);

\draw[edge] (x2) -- (x3);
\draw[edge] (x3) -- (x5);
\draw[edge] (x5) -- (x4);
\draw[edge] (x4) -- (x3);
\draw[edge] (x2) -- (x4);
\draw[hidden] (x2) -- (x5);

\draw[edge] (x1) -- (x2);
\draw[edge] (x1) -- (x4);

\node[vertex] at (x1) {};
\node[vertex] at (x2) {};
\node[vertex] at (x3) {};
\node[vertex] at (x4) {};
\node[vertex] at (x5) {};

\node[left] at (x1) {$x_1$};
\node[left] at (x2) {$x_2$};
\node[above] at (x3) {$x_3$};
\node[below] at (x4) {$x_4$};
\node[right] at (x5) {$x_5$};

\end{tikzpicture}
\caption{A hollow tetrahedron on vertices $x_2,x_3,x_4,x_5$ with a hollow triangle attached along the edge $x_2x_4$.}
    \label{fig:hollow-tetrahedron-attached-triangle}
\end{figure}

Then, the minimal nonfaces of $\Delta$ are $\{1,3\}, \{1,5\}, \{1,2,4\}, \{2,3,4,5\}. $

Using the correspondence $F\mapsto x_F=\prod_{i\in F}x_i$, these become
\[
x_{\{1,3\}}=x_1x_3,\quad
x_{\{1,5\}}=x_1x_5,\quad
x_{\{1,2,4\}}=x_1x_2x_4,\quad
x_{\{2,3,4,5\}}=x_2x_3x_4x_5.
\]
Therefore the Stanley--Reisner ideal is
\[
I_\Delta=(x_1x_3,\ x_1x_5,\ x_1x_2x_4,\ x_2x_3x_4x_5).
\]

Topologically, the boundary of the tetrahedron contributes a 2-dimensional hole, while the attached hollow triangle contributes a 1-dimensional hole. Therefore, $\Delta$ is homotopy equivalent to $S^2 \vee S^1$. Thus,

\[
\widetilde H_2 (\Delta ; k) \cong k, \quad \widetilde H_1 (\Delta ; k) \cong k, \quad \widetilde H_0 (\Delta ; k) = 0.
\]
Taking $W=[5]$ in Hochster's formula, we get
\[
\beta_{i,5}^S(S/I_\Delta)
=
\dim_k\widetilde H_{5-i-1}(\Delta;k).
\]

Therefore,
\[
\beta_{2,5}^S(S/I_\Delta)=1
\]
because $5-2-1=2$, so Hochster uses
\[
\widetilde H_2(\Delta;k)\cong k.
\]
Further,
\[
\beta_{3,5}^S(S/I_\Delta)=1
\]
because $5-3-1=1$, so Hochster uses
\[
\widetilde H_1(\Delta;k)\cong k.
\]

This example demonstrates the bridge between an algebraic property and a topological property through Hochster's formula: topological holes in induced subcomplexes can contribute directly to Betti numbers.
    
\end{exam}

\begin{corollary}[Torsion-Free Criterion]\label{torsionfreecriterion}
Let $\Delta$ be a simplicial complex on vertex set contained in $[n]$. Suppose that for every $W \subseteq [n]$ and every $q$, the group 
\[
\widetilde H_q(\Delta_W;\mathbb Z)
\]
is torsion-free. Then the multigraded Betti numbers, and consequently the graded Betti numbers, of $S/I_\Delta$ are independent of $\operatorname{char}(k)$.
\end{corollary}

\begin{proof}
By multigraded Hochster's formula, for every $W\subseteq[n]$,
\[
\beta_{i,x_W}^S(S/I_\Delta)
=
\dim_k \widetilde H_{|W|-i-1}(\Delta_W;k).
\]
Set $q=|W|-i-1$.

If $q=-1$, then by our convention $\widetilde H_{-1}(\Delta_W;k)$ is either
$0$ or $k$, according as $\Delta_W\neq\{\varnothing\}$ or
$\Delta_W=\{\varnothing\}$. Hence its dimension is independent of
$\operatorname{char}(k)$. If $q<-1$, then
$\widetilde H_q(\Delta_W;k)=0$ by convention.

For $q\geq 0$, Corollary~\ref{UniversalCoefficientTheoremCorollary} shows that
\[
\dim_k \widetilde H_q(\Delta_W;k)
\]
depends only on the rank of $\widetilde H_q(\Delta_W;\mathbb Z)$, and not on the
characteristic of $k$.

Therefore every multigraded Betti number is characteristic-independent. Summing over all $W$ with $|W|=j$ gives characteristic-independence of the graded Betti
numbers.
\end{proof}

\section{The Five-Vertex Topology}
We now use the following topological result which appears as Proposition 3.4 in \cite{govc2025fundamental}. The following small-vertex classification result is the only external topological input in the proof of the main theorem.

\begin{theorem}\label{fivevertextopology}
Every connected simplicial complex on at most five vertices is homotopy equivalent to a wedge of spheres.
\end{theorem}
Here a point is allowed as the empty wedge of spheres.

\begin{corollary}\label{fivevertextorsionfree}
Let $\Delta$ be a simplicial complex whose vertex set is contained in $[n]$, with
$n\leq 5$. Then for every $W\subseteq[n]$ and every $q$,
\[
\widetilde H_q(\Delta_W;\mathbb Z)
\]
is torsion-free.
\end{corollary}

\begin{proof}
Let $U=W\cap V(\Delta)$. Then $\Delta_W=\Delta_U$, since vertices outside $V(\Delta)$ do not belong to any nonempty face of $\Delta$. If $U=\varnothing$, then $\Delta_W=\{\varnothing\}$, so its reduced homology is torsion-free.

Assume $U\neq\varnothing$, and write
\[
\Delta_W=C_1\sqcup\cdots\sqcup C_c
\]
as a disjoint union of connected components. Each $C_\ell$ has at most five vertices, so by Theorem~\ref{fivevertextopology}, each $C_\ell$ is homotopy equivalent to a wedge of spheres. Hence each $\widetilde H_q(C_\ell;\mathbb Z)$ is free abelian. For $q>0$,
\[
\widetilde H_q(\Delta_W;\mathbb Z)
\cong
\bigoplus_{\ell=1}^c \widetilde H_q(C_\ell;\mathbb Z),
\]
and
\[
\widetilde H_0(\Delta_W;\mathbb Z)\cong \mathbb Z^{c-1}.
\]
Thus every reduced integral homology group of $\Delta_W$ is free abelian, hence torsion-free.
\end{proof}

\section{Characteristic-Independence for Monomial Ideals in Five Variables}

We first prove the result for squarefree monomial ideals.

\begin{theorem}\label{squarefreefivevariabletheorem}
Let $S=k[x_1,\dots,x_n]$ with $n\leq 5$, and let $I\subseteq S$ be a squarefree monomial ideal. Then the multigraded and graded Betti numbers of $S/I$ are independent of the characteristic of $k$.
\end{theorem}

\begin{proof}
If $I=S$, then $S/I=0$, so all Betti numbers vanish. Hence assume $I$ is proper. Since $I$ is squarefree, there is a simplicial complex $\Delta$ on a vertex set contained in $[n]$ such that $I=I_\Delta$.

If the multidegree is not squarefree, then the corresponding multigraded Betti number is zero by Hochster's formula. For a squarefree multidegree $x_W$, with $W\subseteq[n]$, Hochster's formula gives
\[
\beta_{i,x_W}^S(S/I_\Delta)
=
\dim_k \widetilde H_{|W|-i-1}(\Delta_W;k).
\]
Since $n\leq 5$, every induced subcomplex $\Delta_W$ has at most five vertices. By Corollary~\ref{fivevertextorsionfree}, every group
\[
\widetilde H_q(\Delta_W;\mathbb Z)
\]
is torsion-free. Therefore, by Corollary~\ref{torsionfreecriterion}, every multigraded Betti number of $S/I$ is independent of $\operatorname{char}(k)$. The graded Betti numbers are obtained by summing multigraded Betti numbers, so they are characteristic-independent as well.
\end{proof}

Next, we extend the result from squarefree monomial ideals to arbitrary monomial ideals, using the squarefree twin ideal reduction from Section~\ref{reductiontosquareidealssection}.

\begin{proof}[Proof of Theorem~\ref{allfivevariabletheorem}]
Let $m=x_1^{\alpha_1}\cdots x_5^{\alpha_5}$ be a Taylor multidegree appearing in the Taylor resolution $\mathbb T_M$ of $S/M$. By the squarefree twin ideal reduction,
\[
\beta_{i,m}^S(S/M)=\beta_{i,y_m}^{T_m}(T_m/M_m''),
\]
where
\[
T_m=k[y_j:\alpha_j>0]
\]
has at most five variables, $M_m''\subseteq T_m$ is the squarefree twin ideal associated to $m$, and
\[
y_m=\prod_{\alpha_j>0}y_j.
\]

If $M_m''=T_m$, then $T_m/M_m''=0$, so
\[
\beta_{i,y_m}^{T_m}(T_m/M_m'')=0
\]
for all $i$. Hence this multigraded Betti number is characteristic-independent. Otherwise, $M_m''$ is a proper squarefree monomial ideal in at most five variables, so Theorem~\ref{squarefreefivevariabletheorem} implies that
\[
\beta_{i,y_m}^{T_m}(T_m/M_m'')
\]
is independent of $\operatorname{char}(k)$. Therefore,
\[
\beta_{i,m}^S(S/M)
\]
is independent of $\operatorname{char}(k)$ for every Taylor multidegree $m$.

If a multidegree does not occur in the Taylor resolution, then the corresponding multigraded Betti number is zero. Hence all multigraded Betti numbers of $S/M$ are characteristic-independent.

Finally, the graded Betti numbers are obtained by summing multigraded Betti numbers of fixed total degree:
\[
\beta_{i,j}^S(S/M)=\sum_{\deg(m)=j}\beta_{i,m}^S(S/M).
\]
Therefore, the graded Betti numbers are characteristic-independent. Summing the graded Betti numbers over all $j$ gives characteristic-independence of the total Betti numbers.
\end{proof}

\section{Characteristic-Dependence for Monomial Ideals in Six Variables}
The previous sections show that monomial ideals in at most five variables have characteristic-independent Betti numbers. The restriction to five variables is sharp: in six variables, Betti numbers of squarefree monomial ideals may depend on the characteristic of the field.

\begin{corollary}\label{sixvariablecorollary}
Let $S=k[x_1,\ldots,x_6]$, and let $M\subseteq S$ be a monomial ideal. Then the multigraded, graded, and total Betti numbers of $S/M$ are the same over all fields of characteristic different from 2. Thus any characteristic-dependence in six variables can only occur in characteristic 2.
\end{corollary}

\begin{proof}
The proof is the same reduction as in Theorem~\ref{allfivevariabletheorem}. For each Taylor multidegree $m$, the squarefree twin ideal $M_m''$ lies in a polynomial ring with at most six variables. Thus the relevant Betti number is computed, by Hochster's formula, from the reduced homology of induced subcomplexes on at most six vertices.

By the theorem of Govc--Marzantowicz--Michalak--Pave{\v{s}}i{\'c} \cite[Theorem~3.6]{govc2025fundamental}, every connected simplicial complex on at most six vertices is either the minimal triangulation of $\mathbb{RP}^2$ or is homotopy equivalent to a wedge of spheres. Wedges of spheres have torsion-free integral homology, while the minimal triangulation of $\mathbb{RP}^2$ has only $\mathbb Z/2$-torsion. Therefore no induced subcomplex on at most six vertices has $p$-torsion for any prime $p\neq 2$.

The Universal Coefficient Theorem then implies that the dimensions of the homology groups appearing in Hochster's formula are the same over all fields of characteristic different from 2. Hence the multigraded Betti numbers are the same over all such fields, and the same is true for graded and total Betti numbers after summing.
\end{proof}

Peeva records the following six-variable example showing that characteristic-dependence can occur for monomial ideals \cite[Example 12.4]{peeva2011graded}. Let
\[
A=k[a,b,c,e,f,h]
\]
and let
\[
B=(abc,\ abf,\ ace,\ ahe,\ ahf,\ bch,\ bhe,\ bef,\ chf,\ cef).
\]
Then the graded Betti numbers of $A/B$ depend on the characteristic of $k$. If
$\operatorname{char}(k)\neq 2$, the nonzero graded Betti numbers of $A/B$ are
\[
\beta_{0,0}=1,\quad
\beta_{1,3}=10,\quad
\beta_{2,4}=15,\quad
\beta_{3,5}=6.
\]
If $\operatorname{char}(k)=2$, then the nonzero graded Betti numbers are
\[
\beta_{0,0}=1,\quad
\beta_{1,3}=10,\quad
\beta_{2,4}=15,\quad
\beta_{3,5}=6,\quad
\beta_{3,6}=1,\quad
\beta_{4,6}=1.
\]
Therefore, the graded Betti numbers and total Betti numbers depend on the characteristic of the field.

This example comes from the six-vertex triangulation of $\mathbb{RP}^2$, due to Reisner's work \cite{reisner1976cohen}. The relevant topological feature is 
\[ 
H_1(\mathbb{RP}^2;\mathbb Z)\cong \mathbb Z/2. 
\] 
Thus, the associated simplicial complex has 2-torsion in integral homology. 

Moreover, the six-variable case is special in a precise sense: due to Govc, Marzantowicz, Michalak, and Pave{\v{s}}i{\'c}, every connected simplicial complex on at most six vertices is either the unique minimal triangulation of $\mathbb{RP}^2$ or is homotopy equivalent to a wedge of spheres \cite[Theorem 3.6]{govc2025fundamental}. Since wedges of spheres have torsion-free integral homology, the only possible source of torsion for a simplicial complex on at most six vertices is the six-vertex triangulation of $\mathbb{RP}^2$. In turn, the only possible torsion for six variables is $\mathbb Z/2$-torsion. For example, characteristic 3-dependence would require $\mathbb Z/3$-torsion somewhere, and that cannot occur on only six vertices \cite{govc2025fundamental}. Other torsion phenomena may become possible in seven or more variables.

By the Universal Coefficient Theorem, this $\mathbb Z/2$-torsion changes the dimensions of homology over fields of characteristic 2, while this contribution disappears over fields of any other characteristic. Hochster's formula then converts this change in homology into a change in Betti numbers. Thus, this example realizes the only possible six-vertex topological obstruction, namely the 2-torsion coming from the minimal triangulation of $\mathbb{RP}^2$.

In sum, for monomial ideals in six variables, any characteristic-dependence of Betti numbers can only distinguish characteristic 2 from characteristics different from 2, while our theorem shows that it cannot occur for monomial ideals in at most five variables.

\section{Enumeration of Squarefree Ideals and Betti Tables} \label{enumerationsection}

The main theorem settles characteristic-independence for all monomial ideals in $k[x_1,\ldots,x_5]$. We also record the following finite enumeration for squarefree monomial ideals in five variables.

Squarefree monomial ideals in $S=k[x_1,\ldots,x_5]$ correspond to antichains in the Boolean lattice $2^{[5]}$. Under the correspondence
\[
F\subseteq [5]\quad \longleftrightarrow \quad x_F=\prod_{i\in F}x_i,
\]
the minimal generators of a squarefree monomial ideal form a collection of pairwise incomparable subsets of $[5]$ and every such antichain determines a squarefree monomial ideal. The symmetric group $S_5$ acts on these antichains
by relabeling variables. If two squarefree monomial ideals lie in the same $S_5$-orbit, then their quotient rings are isomorphic and hence have the same graded Betti table.

\begin{proposition}\label{computed-betti-tables}
Up to relabeling of variables, there are 208 nonzero proper squarefree monomial ideals in $k[x_1,\ldots,x_5]$. Among the quotient rings $S/I$ arising from these ideals, exactly 134 distinct graded Betti tables occur.
\end{proposition}

\begin{proof}
The total number of labeled antichains in $2^{[5]}$ is the Dedekind number 7581 \cite{oeisA000372}. Up to relabeling by $S_5$, there are 210 orbits of antichains \cite{oeisA003182}. These 210 orbits include the empty antichain, corresponding to the zero ideal (0), and the antichain $\{\varnothing\}$, corresponding to the unit ideal (1). Removing these two degenerate cases leaves 208 nonzero proper squarefree monomial ideals up to relabeling.

We generated these representatives computationally as follows. First, a Python script generated all antichains in $2^{[5]}$. Then, for each antichain, the
script applied all permutations in $S_5$ and kept a canonical representative from its relabeling orbit. This produced 210 orbit representatives, agreeing with the known count above. The complete list of these orbit representatives is given in Table~\ref{table:monomial-ideal-index} in Appendix~A. After removing the representatives of (0) and (1), 208 representatives remained.

For each representative $I$, we computed the graded Betti table of $S/I$ over the field $\mathbb F_{67}$. This choice of field is harmless by Theorem~\ref{squarefreefivevariabletheorem}, since the graded Betti numbers of squarefree monomial ideals in five variables are independent of the characteristic. The computation was checked in two independent ways: by computing minimal free resolutions in Macaulay2 \cite{M2} and by evaluating Hochster's formula in Python. Both implementations produced the same 134 distinct graded Betti tables.
\end{proof}

The multiplicities of the 134 graded Betti table types are as follows:
\[
\begin{array}{c|c|c}
\text{Ideals per Betti table} & \text{Number of Betti tables} &
\text{Ideals accounted for} \\
\hline
1 & 96 & 96\\
2 & 18 & 36\\
3 & 11 & 33\\
4 & 3 & 12\\
5 & 5 & 25\\
6 & 1 & 6\\
\hline
& 134 & 208
\end{array}
\]

Thus most graded Betti tables are realized by a unique relabeling class of squarefree monomial ideals, but coincidences do occur. Hochster's formula explains why this is possible: the graded Betti table records only the total dimensions of the reduced homology groups of induced subcomplexes of each fixed cardinality. It does not record which particular induced subcomplexes contribute those homology groups. Therefore, distinct Stanley--Reisner complexes can have the same graded Betti table.

\section{Veronese and Almost Veronese Ideals}

Indeed, some of the computed Betti table types can be explained by topology. Let 
\[
I_{d,n} = (x_F: |F|=d) \subseteq S=k[x_1,\ldots,x_n]
\]
be the squarefree Veronese ideal. Then the Stanley--Reisner complex of $I_{d,n}$ is 
\[
\Delta = \{F \subseteq [n]: |F|\leq d-1\},
\]
which is just the ($d-2$)-skeleton of the ($n-1$)-simplex. Therefore, if $|W|=r$, then $\Delta_W$ is the ($d-2$)-skeleton of an ($r-1$)-simplex. Its reduced homology is concentrated in degree ($d-2$), with rank $\binom{r-1}{d-1}$. Therefore, Hochster's formula gives
\[
\beta_{r-d+1,r}(S/I_{d,n}) = \binom{n}{r}\binom{r-1}{d-1} \quad d\leq r \leq n,
\]
with all other graded Betti numbers equal to zero, except for $\beta_{0,0}=1$. This agrees with the known characteristic-free behavior of squarefree Veronese ideals \cite{galetto2020squarefree}.

\begin{exam}
For the squarefree Veronese ideal $I_{3,5}$, the formula above gives
\[
\beta_{1,3}(S/I_{3,5})=\binom{5}{3}\binom{2}{2}=10,
\]
\[
\beta_{2,4}(S/I_{3,5})=\binom{5}{4}\binom{3}{2}=15,
\]
and
\[
\beta_{3,5}(S/I_{3,5})=\binom{5}{5}\binom{4}{2}=6.
\]
Together with $\beta_{0,0}=1$, these are the only nonzero graded Betti numbers of $S/I_{3,5}$. Thus the Betti table is
\[
\begin{array}{c|cccc}
 & 0 & 1 & 2 & 3\\
\hline
\mathrm{total} & 1 & 10 & 15 & 6\\
0 & 1 & \cdot & \cdot & \cdot\\
1 & \cdot & \cdot & \cdot & \cdot\\
2 & \cdot & 10 & 15 & 6
\end{array}
\]
This Betti table is one of the 134 table types appearing in the five-variable enumeration above.
\end{exam}

A similar explanation applies to the almost squarefree Veronese ideals. Following Jafari, Mafi, and Saremi \cite{jafari2020sequentially}, consider the ideal obtained by deleting one generator from $I_{d,n}$: fix a $d$-subset $F_0\subseteq[n]$, and let 
\[
I_{d,n}^{-} = (x_F: |F|=d,\ F\neq F_0)
\]
be the almost squarefree Veronese ideal. Topologically, the Stanley--Reisner complex of $I_{d,n}^{-}$ is obtained from the ($d-2$)-skeleton by adding the single ($d-1$)-simplex $F_0$. Therefore, for an induced subcomplex on a set $W$ with $|W|=r$, the usual rank $\binom{r-1}{d-1}$ of the top reduced homology of the $(d-2)$-skeleton is decreased by $1$ exactly when $F_0\subseteq W$. There are $\binom{n-d}{r-d}$ subsets $W\subseteq[n]$ of size $r$ containing $F_0$. Hochster's formula therefore gives, for $d\leq r\leq n$,
\[
\beta_{r-d+1,r}(S/I_{d,n}^{-})
=
\binom{n}{r}\binom{r-1}{d-1}
-
\binom{n-d}{r-d}.
\]
All other graded Betti numbers vanish, except for $\beta_{0,0}=1$.

These examples illustrate how some of the computed Betti tables can be explained by familiar topological structures. 

\section{Conclusion and Further Directions}
We proved that monomial ideals in $k[x_1,\ldots,x_5]$ have characteristic-independent multigraded, graded, and total Betti numbers. The proof reduces each multigraded Betti number to one for a squarefree twin ideal and then applies Hochster's formula together with the fact that simplicial complexes on at most five vertices have torsion-free reduced integral homology. The classical six-variable example arising from the six-vertex triangulation of $\mathbb{RP}^{2}$ shows that the result is sharp. Moreover, the six-vertex classification of Govc--Marzantowicz--Michalak--Pave{\v{s}}i{\'c} shows that in six variables the only possible characteristic dependence comes from the $\mathbb Z/2$-torsion of the minimal triangulation of $\mathbb{RP}^2$. Finally, our computation shows that the 208 nonzero proper squarefree monomial ideals in five variables up to relabeling realize 134 distinct graded Betti tables.

Several natural questions remain. First, can the 134 Betti table types be classified by the homology profiles of induced subcomplexes on five vertices? Hochster's formula shows that a graded Betti table records only the total reduced homology dimensions of induced subcomplexes of each cardinality, not which induced subcomplexes contribute. This loss of information explains why nonisomorphic Stanley--Reisner complexes can have the same graded Betti table.

Second, which structural families account for repeated Betti tables? The squarefree Veronese and almost squarefree Veronese ideals give one source of such families, governed by simplex-skeleton homology. It would be interesting to identify the other recurring topological patterns among five-vertex complexes.

Finally, for general $n$, how many distinct graded Betti tables occur among squarefree monomial ideals in $n$ variables up to relabeling? The case $n=5$ gives 134 table types among 208 nonzero proper relabeling classes, suggesting a broader enumerative problem.

\section*{Data and code availability}
The code used to generate the relabeling representatives and compute the graded Betti tables, together with the output files, is available in the accompanying repository: \url{https://github.com/voltroom0606/fivevariablebettinumbers}. AI tools were used to assist with code drafting, formatting, and debugging. The mathematical arguments and computational outputs were checked by the authors, including agreement between the Macaulay2 minimal-resolution computation and an independent Python implementation of Hochster's formula.

\section*{Acknowledgements}
This work was supported by NSF grant DMS-2244020. The authors are grateful to Professor Guillermo Alesandroni for guidance, encouragement, and helpful comments on earlier versions of this manuscript. The authors also thank Christopher Chin, Abigail Ho, and Patrick Welte for useful discussions, and thank the California State University Chico Mathematics and Statistics Department, in particular John Lind, for providing the setting in which this work was carried out. 

\appendix

\section{The 210 relabeling classes of squarefree monomial ideals in five variables}
\label{listofmonomialidealsfivevar}

{\scriptsize
\begin{longtable}{@{}r>{\raggedright\arraybackslash}p{0.86\textwidth}@{}}
\caption{Table of squarefree monomial ideals with 5 variables up to relabeling. }\label{table:monomial-ideal-index}\\
\toprule
Index & Ideal \\
\midrule
\endfirsthead
\toprule
Index & Ideal \\
\midrule
\endhead
\midrule
\multicolumn{2}{r}{\emph{Continued on next page}}\\
\midrule
\endfoot
\bottomrule
\endlastfoot
1 & \(\bigl(0\bigr)\) \\
2 & \(\bigl(1\bigr)\) \\
3 & \(\bigl(x_{1}\bigr)\) \\
4 & \(\bigl(x_{1}x_{2}\bigr)\) \\
5 & \(\bigl(x_{1}x_{2}x_{3}\bigr)\) \\
6 & \(\bigl(x_{1}x_{2}x_{3}x_{4}\bigr)\) \\
7 & \(\bigl(x_{1}x_{2}x_{3}x_{4}x_{5}\bigr)\) \\
8 & \(\bigl(x_{1},\allowbreak x_{2}\bigr)\) \\
9 & \(\bigl(x_{1},\allowbreak x_{2}x_{3}\bigr)\) \\
10 & \(\bigl(x_{1},\allowbreak x_{2}x_{3}x_{4}\bigr)\) \\
11 & \(\bigl(x_{1},\allowbreak x_{2}x_{3}x_{4}x_{5}\bigr)\) \\
12 & \(\bigl(x_{1}x_{2},\allowbreak x_{1}x_{3}\bigr)\) \\
13 & \(\bigl(x_{1}x_{2},\allowbreak x_{3}x_{4}\bigr)\) \\
14 & \(\bigl(x_{1}x_{2},\allowbreak x_{1}x_{3}x_{4}\bigr)\) \\
15 & \(\bigl(x_{1}x_{2},\allowbreak x_{3}x_{4}x_{5}\bigr)\) \\
16 & \(\bigl(x_{1}x_{2},\allowbreak x_{1}x_{3}x_{4}x_{5}\bigr)\) \\
17 & \(\bigl(x_{1}x_{2}x_{3},\allowbreak x_{1}x_{2}x_{4}\bigr)\) \\
18 & \(\bigl(x_{1}x_{2}x_{3},\allowbreak x_{1}x_{4}x_{5}\bigr)\) \\
19 & \(\bigl(x_{1}x_{2}x_{3},\allowbreak x_{1}x_{2}x_{4}x_{5}\bigr)\) \\
20 & \(\bigl(x_{1}x_{2}x_{3}x_{4},\allowbreak x_{1}x_{2}x_{3}x_{5}\bigr)\) \\
21 & \(\bigl(x_{1},\allowbreak x_{2},\allowbreak x_{3}\bigr)\) \\
22 & \(\bigl(x_{1},\allowbreak x_{2},\allowbreak x_{3}x_{4}\bigr)\) \\
23 & \(\bigl(x_{1},\allowbreak x_{2},\allowbreak x_{3}x_{4}x_{5}\bigr)\) \\
24 & \(\bigl(x_{1},\allowbreak x_{2}x_{3},\allowbreak x_{2}x_{4}\bigr)\) \\
25 & \(\bigl(x_{1},\allowbreak x_{2}x_{3},\allowbreak x_{4}x_{5}\bigr)\) \\
26 & \(\bigl(x_{1},\allowbreak x_{2}x_{3},\allowbreak x_{2}x_{4}x_{5}\bigr)\) \\
27 & \(\bigl(x_{1},\allowbreak x_{2}x_{3}x_{4},\allowbreak x_{2}x_{3}x_{5}\bigr)\) \\
28 & \(\bigl(x_{1}x_{2},\allowbreak x_{1}x_{3},\allowbreak x_{2}x_{3}\bigr)\) \\
29 & \(\bigl(x_{1}x_{2},\allowbreak x_{1}x_{3},\allowbreak x_{1}x_{4}\bigr)\) \\
30 & \(\bigl(x_{1}x_{2},\allowbreak x_{1}x_{3},\allowbreak x_{2}x_{4}\bigr)\) \\
31 & \(\bigl(x_{1}x_{2},\allowbreak x_{1}x_{3},\allowbreak x_{2}x_{3}x_{4}\bigr)\) \\
32 & \(\bigl(x_{1}x_{2},\allowbreak x_{1}x_{3},\allowbreak x_{4}x_{5}\bigr)\) \\
33 & \(\bigl(x_{1}x_{2},\allowbreak x_{1}x_{3},\allowbreak x_{1}x_{4}x_{5}\bigr)\) \\
34 & \(\bigl(x_{1}x_{2},\allowbreak x_{1}x_{3},\allowbreak x_{2}x_{4}x_{5}\bigr)\) \\
35 & \(\bigl(x_{1}x_{2},\allowbreak x_{1}x_{3},\allowbreak x_{2}x_{3}x_{4}x_{5}\bigr)\) \\
36 & \(\bigl(x_{1}x_{2},\allowbreak x_{3}x_{4},\allowbreak x_{1}x_{3}x_{5}\bigr)\) \\
37 & \(\bigl(x_{1}x_{2},\allowbreak x_{1}x_{3}x_{4},\allowbreak x_{2}x_{3}x_{4}\bigr)\) \\
38 & \(\bigl(x_{1}x_{2},\allowbreak x_{1}x_{3}x_{4},\allowbreak x_{1}x_{3}x_{5}\bigr)\) \\
39 & \(\bigl(x_{1}x_{2},\allowbreak x_{1}x_{3}x_{4},\allowbreak x_{2}x_{3}x_{5}\bigr)\) \\
40 & \(\bigl(x_{1}x_{2},\allowbreak x_{1}x_{3}x_{4},\allowbreak x_{3}x_{4}x_{5}\bigr)\) \\
41 & \(\bigl(x_{1}x_{2},\allowbreak x_{1}x_{3}x_{4},\allowbreak x_{2}x_{3}x_{4}x_{5}\bigr)\) \\
42 & \(\bigl(x_{1}x_{2},\allowbreak x_{1}x_{3}x_{4}x_{5},\allowbreak x_{2}x_{3}x_{4}x_{5}\bigr)\) \\
43 & \(\bigl(x_{1}x_{2}x_{3},\allowbreak x_{1}x_{2}x_{4},\allowbreak x_{1}x_{3}x_{4}\bigr)\) \\
44 & \(\bigl(x_{1}x_{2}x_{3},\allowbreak x_{1}x_{2}x_{4},\allowbreak x_{1}x_{2}x_{5}\bigr)\) \\
45 & \(\bigl(x_{1}x_{2}x_{3},\allowbreak x_{1}x_{2}x_{4},\allowbreak x_{1}x_{3}x_{5}\bigr)\) \\
46 & \(\bigl(x_{1}x_{2}x_{3},\allowbreak x_{1}x_{2}x_{4},\allowbreak x_{3}x_{4}x_{5}\bigr)\) \\
47 & \(\bigl(x_{1}x_{2}x_{3},\allowbreak x_{1}x_{2}x_{4},\allowbreak x_{1}x_{3}x_{4}x_{5}\bigr)\) \\
48 & \(\bigl(x_{1}x_{2}x_{3},\allowbreak x_{1}x_{4}x_{5},\allowbreak x_{2}x_{3}x_{4}x_{5}\bigr)\) \\
49 & \(\bigl(x_{1}x_{2}x_{3},\allowbreak x_{1}x_{2}x_{4}x_{5},\allowbreak x_{1}x_{3}x_{4}x_{5}\bigr)\) \\
50 & \(\bigl(x_{1}x_{2}x_{3}x_{4},\allowbreak x_{1}x_{2}x_{3}x_{5},\allowbreak x_{1}x_{2}x_{4}x_{5}\bigr)\) \\
51 & \(\bigl(x_{1},\allowbreak x_{2},\allowbreak x_{3},\allowbreak x_{4}\bigr)\) \\
52 & \(\bigl(x_{1},\allowbreak x_{2},\allowbreak x_{3},\allowbreak x_{4}x_{5}\bigr)\) \\
53 & \(\bigl(x_{1},\allowbreak x_{2},\allowbreak x_{3}x_{4},\allowbreak x_{3}x_{5}\bigr)\) \\
54 & \(\bigl(x_{1},\allowbreak x_{2}x_{3},\allowbreak x_{2}x_{4},\allowbreak x_{3}x_{4}\bigr)\) \\
55 & \(\bigl(x_{1},\allowbreak x_{2}x_{3},\allowbreak x_{2}x_{4},\allowbreak x_{2}x_{5}\bigr)\) \\
56 & \(\bigl(x_{1},\allowbreak x_{2}x_{3},\allowbreak x_{2}x_{4},\allowbreak x_{3}x_{5}\bigr)\) \\
57 & \(\bigl(x_{1},\allowbreak x_{2}x_{3},\allowbreak x_{2}x_{4},\allowbreak x_{3}x_{4}x_{5}\bigr)\) \\
58 & \(\bigl(x_{1},\allowbreak x_{2}x_{3},\allowbreak x_{2}x_{4}x_{5},\allowbreak x_{3}x_{4}x_{5}\bigr)\) \\
59 & \(\bigl(x_{1},\allowbreak x_{2}x_{3}x_{4},\allowbreak x_{2}x_{3}x_{5},\allowbreak x_{2}x_{4}x_{5}\bigr)\) \\
60 & \(\bigl(x_{1}x_{2},\allowbreak x_{1}x_{3},\allowbreak x_{2}x_{3},\allowbreak x_{1}x_{4}\bigr)\) \\
61 & \(\bigl(x_{1}x_{2},\allowbreak x_{1}x_{3},\allowbreak x_{2}x_{3},\allowbreak x_{4}x_{5}\bigr)\) \\
62 & \(\bigl(x_{1}x_{2},\allowbreak x_{1}x_{3},\allowbreak x_{2}x_{3},\allowbreak x_{1}x_{4}x_{5}\bigr)\) \\
63 & \(\bigl(x_{1}x_{2},\allowbreak x_{1}x_{3},\allowbreak x_{1}x_{4},\allowbreak x_{2}x_{3}x_{4}\bigr)\) \\
64 & \(\bigl(x_{1}x_{2},\allowbreak x_{1}x_{3},\allowbreak x_{1}x_{4},\allowbreak x_{1}x_{5}\bigr)\) \\
65 & \(\bigl(x_{1}x_{2},\allowbreak x_{1}x_{3},\allowbreak x_{1}x_{4},\allowbreak x_{2}x_{5}\bigr)\) \\
66 & \(\bigl(x_{1}x_{2},\allowbreak x_{1}x_{3},\allowbreak x_{1}x_{4},\allowbreak x_{2}x_{3}x_{5}\bigr)\) \\
67 & \(\bigl(x_{1}x_{2},\allowbreak x_{1}x_{3},\allowbreak x_{1}x_{4},\allowbreak x_{2}x_{3}x_{4}x_{5}\bigr)\) \\
68 & \(\bigl(x_{1}x_{2},\allowbreak x_{1}x_{3},\allowbreak x_{2}x_{4},\allowbreak x_{3}x_{4}\bigr)\) \\
69 & \(\bigl(x_{1}x_{2},\allowbreak x_{1}x_{3},\allowbreak x_{2}x_{4},\allowbreak x_{3}x_{5}\bigr)\) \\
70 & \(\bigl(x_{1}x_{2},\allowbreak x_{1}x_{3},\allowbreak x_{2}x_{4},\allowbreak x_{2}x_{3}x_{5}\bigr)\) \\
71 & \(\bigl(x_{1}x_{2},\allowbreak x_{1}x_{3},\allowbreak x_{2}x_{4},\allowbreak x_{3}x_{4}x_{5}\bigr)\) \\
72 & \(\bigl(x_{1}x_{2},\allowbreak x_{1}x_{3},\allowbreak x_{2}x_{3}x_{4},\allowbreak x_{2}x_{3}x_{5}\bigr)\) \\
73 & \(\bigl(x_{1}x_{2},\allowbreak x_{1}x_{3},\allowbreak x_{2}x_{3}x_{4},\allowbreak x_{4}x_{5}\bigr)\) \\
74 & \(\bigl(x_{1}x_{2},\allowbreak x_{1}x_{3},\allowbreak x_{2}x_{3}x_{4},\allowbreak x_{1}x_{4}x_{5}\bigr)\) \\
75 & \(\bigl(x_{1}x_{2},\allowbreak x_{1}x_{3},\allowbreak x_{2}x_{3}x_{4},\allowbreak x_{2}x_{4}x_{5}\bigr)\) \\
76 & \(\bigl(x_{1}x_{2},\allowbreak x_{1}x_{3},\allowbreak x_{1}x_{4}x_{5},\allowbreak x_{2}x_{4}x_{5}\bigr)\) \\
77 & \(\bigl(x_{1}x_{2},\allowbreak x_{1}x_{3},\allowbreak x_{1}x_{4}x_{5},\allowbreak x_{2}x_{3}x_{4}x_{5}\bigr)\) \\
78 & \(\bigl(x_{1}x_{2},\allowbreak x_{1}x_{3},\allowbreak x_{2}x_{4}x_{5},\allowbreak x_{3}x_{4}x_{5}\bigr)\) \\
79 & \(\bigl(x_{1}x_{2},\allowbreak x_{3}x_{4},\allowbreak x_{1}x_{3}x_{5},\allowbreak x_{2}x_{3}x_{5}\bigr)\) \\
80 & \(\bigl(x_{1}x_{2},\allowbreak x_{3}x_{4},\allowbreak x_{1}x_{3}x_{5},\allowbreak x_{2}x_{4}x_{5}\bigr)\) \\
81 & \(\bigl(x_{1}x_{2},\allowbreak x_{1}x_{3}x_{4},\allowbreak x_{2}x_{3}x_{4},\allowbreak x_{1}x_{3}x_{5}\bigr)\) \\
82 & \(\bigl(x_{1}x_{2},\allowbreak x_{1}x_{3}x_{4},\allowbreak x_{2}x_{3}x_{4},\allowbreak x_{3}x_{4}x_{5}\bigr)\) \\
83 & \(\bigl(x_{1}x_{2},\allowbreak x_{1}x_{3}x_{4},\allowbreak x_{1}x_{3}x_{5},\allowbreak x_{1}x_{4}x_{5}\bigr)\) \\
84 & \(\bigl(x_{1}x_{2},\allowbreak x_{1}x_{3}x_{4},\allowbreak x_{1}x_{3}x_{5},\allowbreak x_{2}x_{4}x_{5}\bigr)\) \\
85 & \(\bigl(x_{1}x_{2},\allowbreak x_{1}x_{3}x_{4},\allowbreak x_{1}x_{3}x_{5},\allowbreak x_{3}x_{4}x_{5}\bigr)\) \\
86 & \(\bigl(x_{1}x_{2},\allowbreak x_{1}x_{3}x_{4},\allowbreak x_{1}x_{3}x_{5},\allowbreak x_{2}x_{3}x_{4}x_{5}\bigr)\) \\
87 & \(\bigl(x_{1}x_{2},\allowbreak x_{1}x_{3}x_{4},\allowbreak x_{2}x_{3}x_{5},\allowbreak x_{3}x_{4}x_{5}\bigr)\) \\
88 & \(\bigl(x_{1}x_{2}x_{3},\allowbreak x_{1}x_{2}x_{4},\allowbreak x_{1}x_{3}x_{4},\allowbreak x_{2}x_{3}x_{4}\bigr)\) \\
89 & \(\bigl(x_{1}x_{2}x_{3},\allowbreak x_{1}x_{2}x_{4},\allowbreak x_{1}x_{3}x_{4},\allowbreak x_{1}x_{2}x_{5}\bigr)\) \\
90 & \(\bigl(x_{1}x_{2}x_{3},\allowbreak x_{1}x_{2}x_{4},\allowbreak x_{1}x_{3}x_{4},\allowbreak x_{2}x_{3}x_{5}\bigr)\) \\
91 & \(\bigl(x_{1}x_{2}x_{3},\allowbreak x_{1}x_{2}x_{4},\allowbreak x_{1}x_{3}x_{4},\allowbreak x_{2}x_{3}x_{4}x_{5}\bigr)\) \\
92 & \(\bigl(x_{1}x_{2}x_{3},\allowbreak x_{1}x_{2}x_{4},\allowbreak x_{1}x_{2}x_{5},\allowbreak x_{3}x_{4}x_{5}\bigr)\) \\
93 & \(\bigl(x_{1}x_{2}x_{3},\allowbreak x_{1}x_{2}x_{4},\allowbreak x_{1}x_{2}x_{5},\allowbreak x_{1}x_{3}x_{4}x_{5}\bigr)\) \\
94 & \(\bigl(x_{1}x_{2}x_{3},\allowbreak x_{1}x_{2}x_{4},\allowbreak x_{1}x_{3}x_{5},\allowbreak x_{1}x_{4}x_{5}\bigr)\) \\
95 & \(\bigl(x_{1}x_{2}x_{3},\allowbreak x_{1}x_{2}x_{4},\allowbreak x_{1}x_{3}x_{5},\allowbreak x_{2}x_{4}x_{5}\bigr)\) \\
96 & \(\bigl(x_{1}x_{2}x_{3},\allowbreak x_{1}x_{2}x_{4},\allowbreak x_{1}x_{3}x_{5},\allowbreak x_{2}x_{3}x_{4}x_{5}\bigr)\) \\
97 & \(\bigl(x_{1}x_{2}x_{3},\allowbreak x_{1}x_{2}x_{4},\allowbreak x_{1}x_{3}x_{4}x_{5},\allowbreak x_{2}x_{3}x_{4}x_{5}\bigr)\) \\
98 & \(\bigl(x_{1}x_{2}x_{3},\allowbreak x_{1}x_{2}x_{4}x_{5},\allowbreak x_{1}x_{3}x_{4}x_{5},\allowbreak x_{2}x_{3}x_{4}x_{5}\bigr)\) \\
99 & \(\bigl(x_{1}x_{2}x_{3}x_{4},\allowbreak x_{1}x_{2}x_{3}x_{5},\allowbreak x_{1}x_{2}x_{4}x_{5},\allowbreak x_{1}x_{3}x_{4}x_{5}\bigr)\) \\
100 & \(\bigl(x_{1},\allowbreak x_{2},\allowbreak x_{3},\allowbreak x_{4},\allowbreak x_{5}\bigr)\) \\
101 & \(\bigl(x_{1},\allowbreak x_{2},\allowbreak x_{3}x_{4},\allowbreak x_{3}x_{5},\allowbreak x_{4}x_{5}\bigr)\) \\
102 & \(\bigl(x_{1},\allowbreak x_{2}x_{3},\allowbreak x_{2}x_{4},\allowbreak x_{3}x_{4},\allowbreak x_{2}x_{5}\bigr)\) \\
103 & \(\bigl(x_{1},\allowbreak x_{2}x_{3},\allowbreak x_{2}x_{4},\allowbreak x_{2}x_{5},\allowbreak x_{3}x_{4}x_{5}\bigr)\) \\
104 & \(\bigl(x_{1},\allowbreak x_{2}x_{3},\allowbreak x_{2}x_{4},\allowbreak x_{3}x_{5},\allowbreak x_{4}x_{5}\bigr)\) \\
105 & \(\bigl(x_{1},\allowbreak x_{2}x_{3}x_{4},\allowbreak x_{2}x_{3}x_{5},\allowbreak x_{2}x_{4}x_{5},\allowbreak x_{3}x_{4}x_{5}\bigr)\) \\
106 & \(\bigl(x_{1}x_{2},\allowbreak x_{1}x_{3},\allowbreak x_{2}x_{3},\allowbreak x_{1}x_{4},\allowbreak x_{2}x_{4}\bigr)\) \\
107 & \(\bigl(x_{1}x_{2},\allowbreak x_{1}x_{3},\allowbreak x_{2}x_{3},\allowbreak x_{1}x_{4},\allowbreak x_{1}x_{5}\bigr)\) \\
108 & \(\bigl(x_{1}x_{2},\allowbreak x_{1}x_{3},\allowbreak x_{2}x_{3},\allowbreak x_{1}x_{4},\allowbreak x_{2}x_{5}\bigr)\) \\
109 & \(\bigl(x_{1}x_{2},\allowbreak x_{1}x_{3},\allowbreak x_{2}x_{3},\allowbreak x_{1}x_{4},\allowbreak x_{4}x_{5}\bigr)\) \\
110 & \(\bigl(x_{1}x_{2},\allowbreak x_{1}x_{3},\allowbreak x_{2}x_{3},\allowbreak x_{1}x_{4},\allowbreak x_{2}x_{4}x_{5}\bigr)\) \\
111 & \(\bigl(x_{1}x_{2},\allowbreak x_{1}x_{3},\allowbreak x_{2}x_{3},\allowbreak x_{1}x_{4}x_{5},\allowbreak x_{2}x_{4}x_{5}\bigr)\) \\
112 & \(\bigl(x_{1}x_{2},\allowbreak x_{1}x_{3},\allowbreak x_{1}x_{4},\allowbreak x_{2}x_{3}x_{4},\allowbreak x_{1}x_{5}\bigr)\) \\
113 & \(\bigl(x_{1}x_{2},\allowbreak x_{1}x_{3},\allowbreak x_{1}x_{4},\allowbreak x_{2}x_{3}x_{4},\allowbreak x_{2}x_{5}\bigr)\) \\
114 & \(\bigl(x_{1}x_{2},\allowbreak x_{1}x_{3},\allowbreak x_{1}x_{4},\allowbreak x_{2}x_{3}x_{4},\allowbreak x_{2}x_{3}x_{5}\bigr)\) \\
115 & \(\bigl(x_{1}x_{2},\allowbreak x_{1}x_{3},\allowbreak x_{1}x_{4},\allowbreak x_{1}x_{5},\allowbreak x_{2}x_{3}x_{4}x_{5}\bigr)\) \\
116 & \(\bigl(x_{1}x_{2},\allowbreak x_{1}x_{3},\allowbreak x_{1}x_{4},\allowbreak x_{2}x_{5},\allowbreak x_{3}x_{5}\bigr)\) \\
117 & \(\bigl(x_{1}x_{2},\allowbreak x_{1}x_{3},\allowbreak x_{1}x_{4},\allowbreak x_{2}x_{5},\allowbreak x_{3}x_{4}x_{5}\bigr)\) \\
118 & \(\bigl(x_{1}x_{2},\allowbreak x_{1}x_{3},\allowbreak x_{1}x_{4},\allowbreak x_{2}x_{3}x_{5},\allowbreak x_{2}x_{4}x_{5}\bigr)\) \\
119 & \(\bigl(x_{1}x_{2},\allowbreak x_{1}x_{3},\allowbreak x_{2}x_{4},\allowbreak x_{3}x_{4},\allowbreak x_{2}x_{3}x_{5}\bigr)\) \\
120 & \(\bigl(x_{1}x_{2},\allowbreak x_{1}x_{3},\allowbreak x_{2}x_{4},\allowbreak x_{3}x_{5},\allowbreak x_{4}x_{5}\bigr)\) \\
121 & \(\bigl(x_{1}x_{2},\allowbreak x_{1}x_{3},\allowbreak x_{2}x_{4},\allowbreak x_{3}x_{5},\allowbreak x_{1}x_{4}x_{5}\bigr)\) \\
122 & \(\bigl(x_{1}x_{2},\allowbreak x_{1}x_{3},\allowbreak x_{2}x_{4},\allowbreak x_{2}x_{3}x_{5},\allowbreak x_{1}x_{4}x_{5}\bigr)\) \\
123 & \(\bigl(x_{1}x_{2},\allowbreak x_{1}x_{3},\allowbreak x_{2}x_{4},\allowbreak x_{2}x_{3}x_{5},\allowbreak x_{3}x_{4}x_{5}\bigr)\) \\
124 & \(\bigl(x_{1}x_{2},\allowbreak x_{1}x_{3},\allowbreak x_{2}x_{3}x_{4},\allowbreak x_{2}x_{3}x_{5},\allowbreak x_{4}x_{5}\bigr)\) \\
125 & \(\bigl(x_{1}x_{2},\allowbreak x_{1}x_{3},\allowbreak x_{2}x_{3}x_{4},\allowbreak x_{2}x_{3}x_{5},\allowbreak x_{1}x_{4}x_{5}\bigr)\) \\
126 & \(\bigl(x_{1}x_{2},\allowbreak x_{1}x_{3},\allowbreak x_{2}x_{3}x_{4},\allowbreak x_{2}x_{3}x_{5},\allowbreak x_{2}x_{4}x_{5}\bigr)\) \\
127 & \(\bigl(x_{1}x_{2},\allowbreak x_{1}x_{3},\allowbreak x_{2}x_{3}x_{4},\allowbreak x_{1}x_{4}x_{5},\allowbreak x_{2}x_{4}x_{5}\bigr)\) \\
128 & \(\bigl(x_{1}x_{2},\allowbreak x_{1}x_{3},\allowbreak x_{2}x_{3}x_{4},\allowbreak x_{2}x_{4}x_{5},\allowbreak x_{3}x_{4}x_{5}\bigr)\) \\
129 & \(\bigl(x_{1}x_{2},\allowbreak x_{1}x_{3},\allowbreak x_{1}x_{4}x_{5},\allowbreak x_{2}x_{4}x_{5},\allowbreak x_{3}x_{4}x_{5}\bigr)\) \\
130 & \(\bigl(x_{1}x_{2},\allowbreak x_{3}x_{4},\allowbreak x_{1}x_{3}x_{5},\allowbreak x_{2}x_{3}x_{5},\allowbreak x_{1}x_{4}x_{5}\bigr)\) \\
131 & \(\bigl(x_{1}x_{2},\allowbreak x_{1}x_{3}x_{4},\allowbreak x_{2}x_{3}x_{4},\allowbreak x_{1}x_{3}x_{5},\allowbreak x_{2}x_{3}x_{5}\bigr)\) \\
132 & \(\bigl(x_{1}x_{2},\allowbreak x_{1}x_{3}x_{4},\allowbreak x_{2}x_{3}x_{4},\allowbreak x_{1}x_{3}x_{5},\allowbreak x_{1}x_{4}x_{5}\bigr)\) \\
133 & \(\bigl(x_{1}x_{2},\allowbreak x_{1}x_{3}x_{4},\allowbreak x_{2}x_{3}x_{4},\allowbreak x_{1}x_{3}x_{5},\allowbreak x_{2}x_{4}x_{5}\bigr)\) \\
134 & \(\bigl(x_{1}x_{2},\allowbreak x_{1}x_{3}x_{4},\allowbreak x_{2}x_{3}x_{4},\allowbreak x_{1}x_{3}x_{5},\allowbreak x_{3}x_{4}x_{5}\bigr)\) \\
135 & \(\bigl(x_{1}x_{2},\allowbreak x_{1}x_{3}x_{4},\allowbreak x_{1}x_{3}x_{5},\allowbreak x_{1}x_{4}x_{5},\allowbreak x_{3}x_{4}x_{5}\bigr)\) \\
136 & \(\bigl(x_{1}x_{2},\allowbreak x_{1}x_{3}x_{4},\allowbreak x_{1}x_{3}x_{5},\allowbreak x_{1}x_{4}x_{5},\allowbreak x_{2}x_{3}x_{4}x_{5}\bigr)\) \\
137 & \(\bigl(x_{1}x_{2},\allowbreak x_{1}x_{3}x_{4},\allowbreak x_{1}x_{3}x_{5},\allowbreak x_{2}x_{4}x_{5},\allowbreak x_{3}x_{4}x_{5}\bigr)\) \\
138 & \(\bigl(x_{1}x_{2}x_{3},\allowbreak x_{1}x_{2}x_{4},\allowbreak x_{1}x_{3}x_{4},\allowbreak x_{2}x_{3}x_{4},\allowbreak x_{1}x_{2}x_{5}\bigr)\) \\
139 & \(\bigl(x_{1}x_{2}x_{3},\allowbreak x_{1}x_{2}x_{4},\allowbreak x_{1}x_{3}x_{4},\allowbreak x_{1}x_{2}x_{5},\allowbreak x_{1}x_{3}x_{5}\bigr)\) \\
140 & \(\bigl(x_{1}x_{2}x_{3},\allowbreak x_{1}x_{2}x_{4},\allowbreak x_{1}x_{3}x_{4},\allowbreak x_{1}x_{2}x_{5},\allowbreak x_{2}x_{3}x_{5}\bigr)\) \\
141 & \(\bigl(x_{1}x_{2}x_{3},\allowbreak x_{1}x_{2}x_{4},\allowbreak x_{1}x_{3}x_{4},\allowbreak x_{1}x_{2}x_{5},\allowbreak x_{3}x_{4}x_{5}\bigr)\) \\
142 & \(\bigl(x_{1}x_{2}x_{3},\allowbreak x_{1}x_{2}x_{4},\allowbreak x_{1}x_{3}x_{4},\allowbreak x_{1}x_{2}x_{5},\allowbreak x_{2}x_{3}x_{4}x_{5}\bigr)\) \\
143 & \(\bigl(x_{1}x_{2}x_{3},\allowbreak x_{1}x_{2}x_{4},\allowbreak x_{1}x_{3}x_{4},\allowbreak x_{2}x_{3}x_{5},\allowbreak x_{2}x_{4}x_{5}\bigr)\) \\
144 & \(\bigl(x_{1}x_{2}x_{3},\allowbreak x_{1}x_{2}x_{4},\allowbreak x_{1}x_{2}x_{5},\allowbreak x_{1}x_{3}x_{4}x_{5},\allowbreak x_{2}x_{3}x_{4}x_{5}\bigr)\) \\
145 & \(\bigl(x_{1}x_{2}x_{3},\allowbreak x_{1}x_{2}x_{4},\allowbreak x_{1}x_{3}x_{5},\allowbreak x_{1}x_{4}x_{5},\allowbreak x_{2}x_{3}x_{4}x_{5}\bigr)\) \\
146 & \(\bigl(x_{1}x_{2}x_{3},\allowbreak x_{1}x_{2}x_{4},\allowbreak x_{1}x_{3}x_{5},\allowbreak x_{2}x_{4}x_{5},\allowbreak x_{3}x_{4}x_{5}\bigr)\) \\
147 & \(\bigl(x_{1}x_{2}x_{3}x_{4},\allowbreak x_{1}x_{2}x_{3}x_{5},\allowbreak x_{1}x_{2}x_{4}x_{5},\allowbreak x_{1}x_{3}x_{4}x_{5},\allowbreak x_{2}x_{3}x_{4}x_{5}\bigr)\) \\
148 & \(\bigl(x_{1},\allowbreak x_{2}x_{3},\allowbreak x_{2}x_{4},\allowbreak x_{3}x_{4},\allowbreak x_{2}x_{5},\allowbreak x_{3}x_{5}\bigr)\) \\
149 & \(\bigl(x_{1}x_{2},\allowbreak x_{1}x_{3},\allowbreak x_{2}x_{3},\allowbreak x_{1}x_{4},\allowbreak x_{2}x_{4},\allowbreak x_{3}x_{4}\bigr)\) \\
150 & \(\bigl(x_{1}x_{2},\allowbreak x_{1}x_{3},\allowbreak x_{2}x_{3},\allowbreak x_{1}x_{4},\allowbreak x_{2}x_{4},\allowbreak x_{1}x_{5}\bigr)\) \\
151 & \(\bigl(x_{1}x_{2},\allowbreak x_{1}x_{3},\allowbreak x_{2}x_{3},\allowbreak x_{1}x_{4},\allowbreak x_{2}x_{4},\allowbreak x_{3}x_{5}\bigr)\) \\
152 & \(\bigl(x_{1}x_{2},\allowbreak x_{1}x_{3},\allowbreak x_{2}x_{3},\allowbreak x_{1}x_{4},\allowbreak x_{2}x_{4},\allowbreak x_{3}x_{4}x_{5}\bigr)\) \\
153 & \(\bigl(x_{1}x_{2},\allowbreak x_{1}x_{3},\allowbreak x_{2}x_{3},\allowbreak x_{1}x_{4},\allowbreak x_{1}x_{5},\allowbreak x_{4}x_{5}\bigr)\) \\
154 & \(\bigl(x_{1}x_{2},\allowbreak x_{1}x_{3},\allowbreak x_{2}x_{3},\allowbreak x_{1}x_{4},\allowbreak x_{1}x_{5},\allowbreak x_{2}x_{4}x_{5}\bigr)\) \\
155 & \(\bigl(x_{1}x_{2},\allowbreak x_{1}x_{3},\allowbreak x_{2}x_{3},\allowbreak x_{1}x_{4},\allowbreak x_{2}x_{5},\allowbreak x_{4}x_{5}\bigr)\) \\
156 & \(\bigl(x_{1}x_{2},\allowbreak x_{1}x_{3},\allowbreak x_{2}x_{3},\allowbreak x_{1}x_{4},\allowbreak x_{2}x_{5},\allowbreak x_{3}x_{4}x_{5}\bigr)\) \\
157 & \(\bigl(x_{1}x_{2},\allowbreak x_{1}x_{3},\allowbreak x_{2}x_{3},\allowbreak x_{1}x_{4},\allowbreak x_{2}x_{4}x_{5},\allowbreak x_{3}x_{4}x_{5}\bigr)\) \\
158 & \(\bigl(x_{1}x_{2},\allowbreak x_{1}x_{3},\allowbreak x_{2}x_{3},\allowbreak x_{1}x_{4}x_{5},\allowbreak x_{2}x_{4}x_{5},\allowbreak x_{3}x_{4}x_{5}\bigr)\) \\
159 & \(\bigl(x_{1}x_{2},\allowbreak x_{1}x_{3},\allowbreak x_{1}x_{4},\allowbreak x_{2}x_{3}x_{4},\allowbreak x_{1}x_{5},\allowbreak x_{2}x_{3}x_{5}\bigr)\) \\
160 & \(\bigl(x_{1}x_{2},\allowbreak x_{1}x_{3},\allowbreak x_{1}x_{4},\allowbreak x_{2}x_{3}x_{4},\allowbreak x_{2}x_{5},\allowbreak x_{3}x_{5}\bigr)\) \\
161 & \(\bigl(x_{1}x_{2},\allowbreak x_{1}x_{3},\allowbreak x_{1}x_{4},\allowbreak x_{2}x_{3}x_{4},\allowbreak x_{2}x_{5},\allowbreak x_{3}x_{4}x_{5}\bigr)\) \\
162 & \(\bigl(x_{1}x_{2},\allowbreak x_{1}x_{3},\allowbreak x_{1}x_{4},\allowbreak x_{2}x_{3}x_{4},\allowbreak x_{2}x_{3}x_{5},\allowbreak x_{2}x_{4}x_{5}\bigr)\) \\
163 & \(\bigl(x_{1}x_{2},\allowbreak x_{1}x_{3},\allowbreak x_{1}x_{4},\allowbreak x_{2}x_{5},\allowbreak x_{3}x_{5},\allowbreak x_{4}x_{5}\bigr)\) \\
164 & \(\bigl(x_{1}x_{2},\allowbreak x_{1}x_{3},\allowbreak x_{1}x_{4},\allowbreak x_{2}x_{3}x_{5},\allowbreak x_{2}x_{4}x_{5},\allowbreak x_{3}x_{4}x_{5}\bigr)\) \\
165 & \(\bigl(x_{1}x_{2},\allowbreak x_{1}x_{3},\allowbreak x_{2}x_{4},\allowbreak x_{3}x_{4},\allowbreak x_{2}x_{3}x_{5},\allowbreak x_{1}x_{4}x_{5}\bigr)\) \\
166 & \(\bigl(x_{1}x_{2},\allowbreak x_{1}x_{3},\allowbreak x_{2}x_{4},\allowbreak x_{2}x_{3}x_{5},\allowbreak x_{1}x_{4}x_{5},\allowbreak x_{3}x_{4}x_{5}\bigr)\) \\
167 & \(\bigl(x_{1}x_{2},\allowbreak x_{1}x_{3},\allowbreak x_{2}x_{3}x_{4},\allowbreak x_{2}x_{3}x_{5},\allowbreak x_{1}x_{4}x_{5},\allowbreak x_{2}x_{4}x_{5}\bigr)\) \\
168 & \(\bigl(x_{1}x_{2},\allowbreak x_{1}x_{3},\allowbreak x_{2}x_{3}x_{4},\allowbreak x_{2}x_{3}x_{5},\allowbreak x_{2}x_{4}x_{5},\allowbreak x_{3}x_{4}x_{5}\bigr)\) \\
169 & \(\bigl(x_{1}x_{2},\allowbreak x_{1}x_{3},\allowbreak x_{2}x_{3}x_{4},\allowbreak x_{1}x_{4}x_{5},\allowbreak x_{2}x_{4}x_{5},\allowbreak x_{3}x_{4}x_{5}\bigr)\) \\
170 & \(\bigl(x_{1}x_{2},\allowbreak x_{3}x_{4},\allowbreak x_{1}x_{3}x_{5},\allowbreak x_{2}x_{3}x_{5},\allowbreak x_{1}x_{4}x_{5},\allowbreak x_{2}x_{4}x_{5}\bigr)\) \\
171 & \(\bigl(x_{1}x_{2},\allowbreak x_{1}x_{3}x_{4},\allowbreak x_{2}x_{3}x_{4},\allowbreak x_{1}x_{3}x_{5},\allowbreak x_{2}x_{3}x_{5},\allowbreak x_{1}x_{4}x_{5}\bigr)\) \\
172 & \(\bigl(x_{1}x_{2},\allowbreak x_{1}x_{3}x_{4},\allowbreak x_{2}x_{3}x_{4},\allowbreak x_{1}x_{3}x_{5},\allowbreak x_{2}x_{3}x_{5},\allowbreak x_{3}x_{4}x_{5}\bigr)\) \\
173 & \(\bigl(x_{1}x_{2},\allowbreak x_{1}x_{3}x_{4},\allowbreak x_{2}x_{3}x_{4},\allowbreak x_{1}x_{3}x_{5},\allowbreak x_{1}x_{4}x_{5},\allowbreak x_{3}x_{4}x_{5}\bigr)\) \\
174 & \(\bigl(x_{1}x_{2},\allowbreak x_{1}x_{3}x_{4},\allowbreak x_{2}x_{3}x_{4},\allowbreak x_{1}x_{3}x_{5},\allowbreak x_{2}x_{4}x_{5},\allowbreak x_{3}x_{4}x_{5}\bigr)\) \\
175 & \(\bigl(x_{1}x_{2}x_{3},\allowbreak x_{1}x_{2}x_{4},\allowbreak x_{1}x_{3}x_{4},\allowbreak x_{2}x_{3}x_{4},\allowbreak x_{1}x_{2}x_{5},\allowbreak x_{1}x_{3}x_{5}\bigr)\) \\
176 & \(\bigl(x_{1}x_{2}x_{3},\allowbreak x_{1}x_{2}x_{4},\allowbreak x_{1}x_{3}x_{4},\allowbreak x_{2}x_{3}x_{4},\allowbreak x_{1}x_{2}x_{5},\allowbreak x_{3}x_{4}x_{5}\bigr)\) \\
177 & \(\bigl(x_{1}x_{2}x_{3},\allowbreak x_{1}x_{2}x_{4},\allowbreak x_{1}x_{3}x_{4},\allowbreak x_{1}x_{2}x_{5},\allowbreak x_{1}x_{3}x_{5},\allowbreak x_{1}x_{4}x_{5}\bigr)\) \\
178 & \(\bigl(x_{1}x_{2}x_{3},\allowbreak x_{1}x_{2}x_{4},\allowbreak x_{1}x_{3}x_{4},\allowbreak x_{1}x_{2}x_{5},\allowbreak x_{1}x_{3}x_{5},\allowbreak x_{2}x_{4}x_{5}\bigr)\) \\
179 & \(\bigl(x_{1}x_{2}x_{3},\allowbreak x_{1}x_{2}x_{4},\allowbreak x_{1}x_{3}x_{4},\allowbreak x_{1}x_{2}x_{5},\allowbreak x_{1}x_{3}x_{5},\allowbreak x_{2}x_{3}x_{4}x_{5}\bigr)\) \\
180 & \(\bigl(x_{1}x_{2}x_{3},\allowbreak x_{1}x_{2}x_{4},\allowbreak x_{1}x_{3}x_{4},\allowbreak x_{1}x_{2}x_{5},\allowbreak x_{2}x_{3}x_{5},\allowbreak x_{3}x_{4}x_{5}\bigr)\) \\
181 & \(\bigl(x_{1}x_{2}x_{3},\allowbreak x_{1}x_{2}x_{4},\allowbreak x_{1}x_{3}x_{4},\allowbreak x_{2}x_{3}x_{5},\allowbreak x_{2}x_{4}x_{5},\allowbreak x_{3}x_{4}x_{5}\bigr)\) \\
182 & \(\bigl(x_{1},\allowbreak x_{2}x_{3},\allowbreak x_{2}x_{4},\allowbreak x_{3}x_{4},\allowbreak x_{2}x_{5},\allowbreak x_{3}x_{5},\allowbreak x_{4}x_{5}\bigr)\) \\
183 & \(\bigl(x_{1}x_{2},\allowbreak x_{1}x_{3},\allowbreak x_{2}x_{3},\allowbreak x_{1}x_{4},\allowbreak x_{2}x_{4},\allowbreak x_{3}x_{4},\allowbreak x_{1}x_{5}\bigr)\) \\
184 & \(\bigl(x_{1}x_{2},\allowbreak x_{1}x_{3},\allowbreak x_{2}x_{3},\allowbreak x_{1}x_{4},\allowbreak x_{2}x_{4},\allowbreak x_{1}x_{5},\allowbreak x_{2}x_{5}\bigr)\) \\
185 & \(\bigl(x_{1}x_{2},\allowbreak x_{1}x_{3},\allowbreak x_{2}x_{3},\allowbreak x_{1}x_{4},\allowbreak x_{2}x_{4},\allowbreak x_{1}x_{5},\allowbreak x_{3}x_{5}\bigr)\) \\
186 & \(\bigl(x_{1}x_{2},\allowbreak x_{1}x_{3},\allowbreak x_{2}x_{3},\allowbreak x_{1}x_{4},\allowbreak x_{2}x_{4},\allowbreak x_{1}x_{5},\allowbreak x_{3}x_{4}x_{5}\bigr)\) \\
187 & \(\bigl(x_{1}x_{2},\allowbreak x_{1}x_{3},\allowbreak x_{2}x_{3},\allowbreak x_{1}x_{4},\allowbreak x_{2}x_{4},\allowbreak x_{3}x_{5},\allowbreak x_{4}x_{5}\bigr)\) \\
188 & \(\bigl(x_{1}x_{2},\allowbreak x_{1}x_{3},\allowbreak x_{2}x_{3},\allowbreak x_{1}x_{4},\allowbreak x_{1}x_{5},\allowbreak x_{2}x_{4}x_{5},\allowbreak x_{3}x_{4}x_{5}\bigr)\) \\
189 & \(\bigl(x_{1}x_{2},\allowbreak x_{1}x_{3},\allowbreak x_{1}x_{4},\allowbreak x_{2}x_{3}x_{4},\allowbreak x_{1}x_{5},\allowbreak x_{2}x_{3}x_{5},\allowbreak x_{2}x_{4}x_{5}\bigr)\) \\
190 & \(\bigl(x_{1}x_{2},\allowbreak x_{1}x_{3},\allowbreak x_{1}x_{4},\allowbreak x_{2}x_{3}x_{4},\allowbreak x_{2}x_{5},\allowbreak x_{3}x_{5},\allowbreak x_{4}x_{5}\bigr)\) \\
191 & \(\bigl(x_{1}x_{2},\allowbreak x_{1}x_{3},\allowbreak x_{1}x_{4},\allowbreak x_{2}x_{3}x_{4},\allowbreak x_{2}x_{3}x_{5},\allowbreak x_{2}x_{4}x_{5},\allowbreak x_{3}x_{4}x_{5}\bigr)\) \\
192 & \(\bigl(x_{1}x_{2},\allowbreak x_{1}x_{3},\allowbreak x_{2}x_{3}x_{4},\allowbreak x_{2}x_{3}x_{5},\allowbreak x_{1}x_{4}x_{5},\allowbreak x_{2}x_{4}x_{5},\allowbreak x_{3}x_{4}x_{5}\bigr)\) \\
193 & \(\bigl(x_{1}x_{2},\allowbreak x_{1}x_{3}x_{4},\allowbreak x_{2}x_{3}x_{4},\allowbreak x_{1}x_{3}x_{5},\allowbreak x_{2}x_{3}x_{5},\allowbreak x_{1}x_{4}x_{5},\allowbreak x_{2}x_{4}x_{5}\bigr)\) \\
194 & \(\bigl(x_{1}x_{2},\allowbreak x_{1}x_{3}x_{4},\allowbreak x_{2}x_{3}x_{4},\allowbreak x_{1}x_{3}x_{5},\allowbreak x_{2}x_{3}x_{5},\allowbreak x_{1}x_{4}x_{5},\allowbreak x_{3}x_{4}x_{5}\bigr)\) \\
195 & \(\bigl(x_{1}x_{2}x_{3},\allowbreak x_{1}x_{2}x_{4},\allowbreak x_{1}x_{3}x_{4},\allowbreak x_{2}x_{3}x_{4},\allowbreak x_{1}x_{2}x_{5},\allowbreak x_{1}x_{3}x_{5},\allowbreak x_{2}x_{3}x_{5}\bigr)\) \\
196 & \(\bigl(x_{1}x_{2}x_{3},\allowbreak x_{1}x_{2}x_{4},\allowbreak x_{1}x_{3}x_{4},\allowbreak x_{2}x_{3}x_{4},\allowbreak x_{1}x_{2}x_{5},\allowbreak x_{1}x_{3}x_{5},\allowbreak x_{1}x_{4}x_{5}\bigr)\) \\
197 & \(\bigl(x_{1}x_{2}x_{3},\allowbreak x_{1}x_{2}x_{4},\allowbreak x_{1}x_{3}x_{4},\allowbreak x_{2}x_{3}x_{4},\allowbreak x_{1}x_{2}x_{5},\allowbreak x_{1}x_{3}x_{5},\allowbreak x_{2}x_{4}x_{5}\bigr)\) \\
198 & \(\bigl(x_{1}x_{2}x_{3},\allowbreak x_{1}x_{2}x_{4},\allowbreak x_{1}x_{3}x_{4},\allowbreak x_{1}x_{2}x_{5},\allowbreak x_{1}x_{3}x_{5},\allowbreak x_{1}x_{4}x_{5},\allowbreak x_{2}x_{3}x_{4}x_{5}\bigr)\) \\
199 & \(\bigl(x_{1}x_{2}x_{3},\allowbreak x_{1}x_{2}x_{4},\allowbreak x_{1}x_{3}x_{4},\allowbreak x_{1}x_{2}x_{5},\allowbreak x_{1}x_{3}x_{5},\allowbreak x_{2}x_{4}x_{5},\allowbreak x_{3}x_{4}x_{5}\bigr)\) \\
200 & \(\bigl(x_{1}x_{2},\allowbreak x_{1}x_{3},\allowbreak x_{2}x_{3},\allowbreak x_{1}x_{4},\allowbreak x_{2}x_{4},\allowbreak x_{3}x_{4},\allowbreak x_{1}x_{5},\allowbreak x_{2}x_{5}\bigr)\) \\
201 & \(\bigl(x_{1}x_{2},\allowbreak x_{1}x_{3},\allowbreak x_{2}x_{3},\allowbreak x_{1}x_{4},\allowbreak x_{2}x_{4},\allowbreak x_{1}x_{5},\allowbreak x_{2}x_{5},\allowbreak x_{3}x_{4}x_{5}\bigr)\) \\
202 & \(\bigl(x_{1}x_{2},\allowbreak x_{1}x_{3},\allowbreak x_{2}x_{3},\allowbreak x_{1}x_{4},\allowbreak x_{2}x_{4},\allowbreak x_{1}x_{5},\allowbreak x_{3}x_{5},\allowbreak x_{4}x_{5}\bigr)\) \\
203 & \(\bigl(x_{1}x_{2},\allowbreak x_{1}x_{3},\allowbreak x_{1}x_{4},\allowbreak x_{2}x_{3}x_{4},\allowbreak x_{1}x_{5},\allowbreak x_{2}x_{3}x_{5},\allowbreak x_{2}x_{4}x_{5},\allowbreak x_{3}x_{4}x_{5}\bigr)\) \\
204 & \(\bigl(x_{1}x_{2},\allowbreak x_{1}x_{3}x_{4},\allowbreak x_{2}x_{3}x_{4},\allowbreak x_{1}x_{3}x_{5},\allowbreak x_{2}x_{3}x_{5},\allowbreak x_{1}x_{4}x_{5},\allowbreak x_{2}x_{4}x_{5},\allowbreak x_{3}x_{4}x_{5}\bigr)\) \\
205 & \(\bigl(x_{1}x_{2}x_{3},\allowbreak x_{1}x_{2}x_{4},\allowbreak x_{1}x_{3}x_{4},\allowbreak x_{2}x_{3}x_{4},\allowbreak x_{1}x_{2}x_{5},\allowbreak x_{1}x_{3}x_{5},\allowbreak x_{2}x_{3}x_{5},\allowbreak x_{1}x_{4}x_{5}\bigr)\) \\
206 & \(\bigl(x_{1}x_{2}x_{3},\allowbreak x_{1}x_{2}x_{4},\allowbreak x_{1}x_{3}x_{4},\allowbreak x_{2}x_{3}x_{4},\allowbreak x_{1}x_{2}x_{5},\allowbreak x_{1}x_{3}x_{5},\allowbreak x_{2}x_{4}x_{5},\allowbreak x_{3}x_{4}x_{5}\bigr)\) \\
207 & \(\bigl(x_{1}x_{2},\allowbreak x_{1}x_{3},\allowbreak x_{2}x_{3},\allowbreak x_{1}x_{4},\allowbreak x_{2}x_{4},\allowbreak x_{3}x_{4},\allowbreak x_{1}x_{5},\allowbreak x_{2}x_{5},\allowbreak x_{3}x_{5}\bigr)\) \\
208 & \(\bigl(x_{1}x_{2}x_{3},\allowbreak x_{1}x_{2}x_{4},\allowbreak x_{1}x_{3}x_{4},\allowbreak x_{2}x_{3}x_{4},\allowbreak x_{1}x_{2}x_{5},\allowbreak x_{1}x_{3}x_{5},\allowbreak x_{2}x_{3}x_{5},\allowbreak x_{1}x_{4}x_{5},\allowbreak x_{2}x_{4}x_{5}\bigr)\) \\
209 & \(\bigl(x_{1}x_{2},\allowbreak x_{1}x_{3},\allowbreak x_{2}x_{3},\allowbreak x_{1}x_{4},\allowbreak x_{2}x_{4},\allowbreak x_{3}x_{4},\allowbreak x_{1}x_{5},\allowbreak x_{2}x_{5},\allowbreak x_{3}x_{5},\allowbreak x_{4}x_{5}\bigr)\) \\
210 & \(\bigl(x_{1}x_{2}x_{3},\allowbreak x_{1}x_{2}x_{4},\allowbreak x_{1}x_{3}x_{4},\allowbreak x_{2}x_{3}x_{4},\allowbreak x_{1}x_{2}x_{5},\allowbreak x_{1}x_{3}x_{5},\allowbreak x_{2}x_{3}x_{5},\allowbreak x_{1}x_{4}x_{5},\allowbreak x_{2}x_{4}x_{5},\allowbreak x_{3}x_{4}x_{5}\bigr)\) \\
\end{longtable}
}

\bibliographystyle{amsalpha}
\bibliography{references}

\end{document}